\documentclass[12pt,reqno]{amsart}
\usepackage{amssymb,amsmath,amsthm,enumerate,bbm}
\usepackage[a4paper]{geometry}

\DeclareMathOperator{\Ran}{Ran}
\DeclareMathOperator{\Dom}{Dom}
\DeclareMathOperator{\Ker}{Ker}

\DeclareMathOperator{\rank}{rank}
\DeclareMathOperator{\BMOA}{BMOA}
\DeclareMathOperator{\Graph}{Graph}

\renewcommand\Im{\hbox{{\rm Im}}\,}
\renewcommand\Re{\hbox{{\rm Re}}\,}
\newcommand{\abs}[1]{\lvert#1\rvert}

\newcommand{\norm}[1]{\lVert#1\rVert}

\newcommand{\jap}[1]{\langle#1\rangle}

\newcommand{\bbT}{{\mathbb T}}
\newcommand{\bbR}{{\mathbb R}}

\newcommand{\bbD}{{\mathbb D}}
\newcommand{\bbP}{\mathbb {P}}

\newcommand{\calH}{{\mathcal H}}
\newcommand{\calP}{\mathcal{P}}

\numberwithin{equation}{section}

\theoremstyle{plain}
\newtheorem{theorem}{\bf Theorem}[section]
\newtheorem*{theorem*}{Theorem 1.1$'$}
\newtheorem{lemma}[theorem]{\bf Lemma}
\newtheorem{proposition}[theorem]{\bf Proposition}
\newtheorem{corollary}[theorem]{\bf Corollary}

\theoremstyle{remark}
\newtheorem*{remark*}{\bf Remark}

\newcommand{\wt}{\widetilde}
\newcommand{\wh}{\widehat}

\renewcommand{\sharp}{\#}

\newcommand{\eps}{\varepsilon}
\newcommand{\1}{\mathbbm{1}}

\begin{document}

\title[Unbounded Hankel operators and the cubic Szeg\H{o} equation]{Unbounded Hankel operators and the flow of the cubic Szeg\H{o} equation}

\author{Patrick G\'erard}
\address{Universit\'e Paris-Saclay, Laboratoire de Math\'ematiques d'Orsay, CNRS, UMR 8628, France}
\email{patrick.gerard@universite-paris-saclay.fr}

\author{Alexander Pushnitski}
\address{Department of Mathematics, King's College London, Strand, London, WC2R~2LS, U.K.}
\email{alexander.pushnitski@kcl.ac.uk}



\begin{abstract}
We prove that, for any Hankel operator with a symbol from the Hardy class $H^2$, the maximal and minimal domains coincide. As an application, we prove that the evolution flow of the cubic Szeg\H{o} equation on the unit circle can be continuously extended to the whole class $H^2$. 
\end{abstract}

%\date{November 2022}

\maketitle

%%%%%%%%%%%%%%%%%%%%%%%%%%%%%%%%%%%%%%%%%%%%%%%%%
%%%%%%%%%%%%%%%%%%%%%%%%%%%%%%%%%%%%%%%%%%%%%%%%%
\section{Introduction and main results}\label{sec.a}
%%%%%%%%%%%%%%%%%%%%%%%%%%%%%%%%%%%%%%%%%%%%%%%%%
%%%%%%%%%%%%%%%%%%%%%%%%%%%%%%%%%%%%%%%%%%%%%%%%%

%%%%%%%%%%%%%%%%%%%%%%%%%%
\subsection{Unbounded Hankel operators}\label{sec.a1}
%%%%%%%%%%%%%%%%%%%%%%%%%%
For $1\leq p\leq \infty$, let  $H^p$ be the standard Hardy class on the unit circle $\bbT$. 
Let $\bbP$ be the orthogonal projection from $L^2(\bbT)$ onto $H^2$. The inner product (linear in the first argument, anti-linear in the second one) on a Hilbert space $X$ is denoted by $\jap{\cdot,\cdot}_X$, and the corresponding norm by $\norm{\cdot}_X$.

Let $u\in H^\infty$; the Hankel operator $\Gamma_u$ with the \emph{symbol} $u$ is defined as the operator on $H^2$ given by 
\begin{equation}
\Gamma_u f=\bbP(u\cdot Jf), \quad f\in H^2, \quad Jf(z)=f(\overline{z}).
\label{a1}
\end{equation}
The matrix of $\Gamma_u$ in the standard basis $\{z^n\}_{n=0}^\infty$ is a Hankel matrix:
\begin{equation}
\jap{\Gamma_u z^n,z^m}_{H^2}=\jap{\bbP(u\overline{z}^n),z^m}_{H^2}=
\jap{u\overline{z}^n,z^m}_{L^2}=\jap{u,z^{n+m}}_{L^2}=\wh u_{n+m},
\label{a1a}
\end{equation}
where $n,m\geq0$ and $\wh u_n$ is the $n$'th Fourier coefficient of $u$. 

In this paper, we are interested in the (potentially unbounded) Hankel operators with symbols $u\in H^2$. In this case, one can still define $\Gamma_u f$ by \eqref{a1} on polynomials $f$. 
The classical Nehari-Fefferman Theorem \cite[Theorem 1.1.2]{Peller} asserts that $\Gamma_u$ extends to a bounded operator on $H^2$ if and only if $u$ is in the $\BMOA$ class. If $u$ is not in $\BMOA$, one can define a \emph{closed} unbounded operator $\Gamma_u$ in at least two \emph{a priori} different ways:
\begin{enumerate}[\rm (i)]
\item
Define $\Gamma_u$ by \eqref{a1} initially on polynomials and take the closure of this operator (it is easy to see that the closure exists);
\item
Define $\Gamma_u$ by \eqref{a1} on the maximal domain, viz.
\begin{equation}
\Dom \Gamma_u=\{f\in H^2: \Gamma_u f\in H^2\}.
\label{a2}
\end{equation}
\end{enumerate}
Here condition $\Gamma_u f\in H^2$ should be understood as follows: 
for the product $h:=u\cdot Jf\in L^1(\bbT)$, the Fourier coefficients $\wh h_n$ with $n\geq0$ are square-summable; in other words, $\Gamma_uf=\bbP(uJf)\in H^2$ with $\bbP$ taken in the sense of distributions.

Our first main result below is that the two definitions (i) and (ii) produce \emph{the same} Hankel operator $\Gamma_u$. In other words, if we define $\Gamma_u$ as in (ii), then the set of all polynomials is dense in $\Dom \Gamma_u$ with respect to the graph norm
\begin{equation}
\norm{f}_{\Gamma_u}:=\bigl(\norm{f}_{H^2}^2+\norm{\Gamma_uf}_{H^2}^2\bigr)^{1/2}.
\label{a3}
\end{equation}
We also have a description of the adjoint to $\Gamma_u$. 
First we recall the description of the adjoint to a bounded Hankel operator. 
For $f\in H^2$, we denote by $f^\sharp\in H^2$ the function $f^\sharp(z)=\overline{f(\overline{z})}$; in other words, the operation $f\mapsto f^\sharp$ is effected by the complex conjugation of all Fourier coefficients of $f$. 
From \eqref{a1a} it is clear that the adjoint of a bounded Hankel operator $\Gamma_u$ is also a Hankel operator given by
$$
\Gamma_u^*=\Gamma_{u^\sharp}.
$$

%%%%%%%%%%%%%%%%%
\begin{theorem}\label{thm.a1}
%%%%%%%%%%%%%%%%%
Let $u\in H^2$, and let $\Gamma_u$ be the Hankel operator, defined by \eqref{a1} on the domain \eqref{a2}. Then: 
\begin{enumerate}[\rm (i)]
\item
$\Gamma_u$ is closed (i.e. $\Dom \Gamma_u$ is closed with respect to the graph norm \eqref{a3});
\item 
the set of all polynomials is dense in $\Dom \Gamma_u$ with respect to the graph norm;
\item 
the adjoint of $\Gamma_u$ is given by $\Gamma_u^*=\Gamma_{u^\sharp}$, where $\Gamma_{u^\sharp}$ is defined as above with $u$ replaced by $u^\sharp$. 
\end{enumerate}
\end{theorem}

%%%%%%%%%%%%%%%%%
\begin{corollary}\label{cr.a1}
%%%%%%%%%%%%%%%%%
Let $u\in H^2$ be a symbol such that all Fourier coefficients of $u$ are real. Then $\Gamma_u$ is essentially self-adjoint on the set of polynomials. 
\end{corollary}

Theorem~\ref{thm.a1} is proved in Section~\ref{sec.b}.
The crucial part of the theorem is part (ii). 
Our  proof of this part is based on the following observation. The graph of $\Gamma_u$, as a subspace of $H^2\oplus H^2$, is an invariant subspace of the operator $S\oplus S^*$, where $S$ is the shift operator in $H^2$ (see Section~\ref{sec.a4}). Invariant subspaces of $S\oplus S^*$ can be characterised in the framework of the Sz.-Nagy--Foias theory of contractions; we borrow this characterisation from Timotin's paper \cite{Timotin}. This characterisation allows us to to prove that the subspace corresponding to polynomials is dense in the graph of $\Gamma_u$.

\subsection{Discussion}
Observe that $\Gamma_u\1=\bbP(u)=u$, where $\1$ is the function in $H^2$ identically equal to one. This shows that if we wish $\1$ to be in the domain of $\Gamma_u$, we must have $u\in H^2$. Thus, the class of symbols $H^2$ is optimal in this sense.

If $\Gamma_u$ is self-adjoint and positive semi-definite (in the quadratic form sense), it is possible to define $\Gamma_u$ via the corresponding quadratic form. This approach has been explored in \cite{Yafaev1,Yafaev2}.

Let us discuss three classes of operators where the question analogous to Theorem~\ref{thm.a1} is relevant.
First of all, there is a close analogy between Corollary~\ref{cr.a1} and the classical problem of self-adjointness of the Schr\"odinger operator $-\Delta+V$ on $L^2(\bbR^n)$, $n>1$ with the potential $0\leq V\in L^2(\bbR^n)$. Condition $V\in L^2$ is the minimal condition under which the expression $-\Delta f+Vf$ is in $L^2$ for all $f\in C^\infty_0(\bbR^n)$. The classical theorem due to Kato  asserts that the operator $-\Delta+V$ is essentially self-adjoint on $C^\infty_0(\bbR^n)$. This is a deep result that relies on Kato's inequality (see e.g. \cite{RS2}).

Another example is offered by the class of Jacobi matrices,
$$ 
J=   
\begin{pmatrix}
b_0& a_0 & 0 &0&\cdots
\\
a_0& b_1 & a_1&0&\cdots
\\
0&a_1&b_2&a_2&\cdots
\\
\vdots&\vdots&\vdots&\vdots&\ddots
\end{pmatrix}\ ,
$$
where $a_n$ are positive numbers and $b_n$ are real numbers. If the sequences $a_n$ and $b_n$ are bounded, $J$ is a bounded operator on $\ell ^2$. On the other hand, if at least one of the sequences $a_n$, $b_n$ is unbounded, the question arises whether $J_{\rm max}=J_{\rm min}$, where 
$J_{\rm min}$ is the closure of the operator $J$ defined on the space of finitely supported sequences, while $J_{\rm max}$ has the domain 
$$ \Dom (J_{\rm max})=\{ (x_n)\in \ell ^2: Jx \in \ell ^2   \} $$
and $Jx$ means the usual matrix product. It may happen that  $J_{\rm max}\ne J_{\rm min}$; 
this is related to the indeterminancy of the corresponding moment problem, see \cite{Simon}.

Finally, let us discuss the closely related case of Toeplitz operators on $H^2$, defined as 
$$T_\varphi f =\mathbb P (\varphi f)\ ,\ \varphi \in L^2(\bbT )\ .$$
 We will focus on the case of   real valued $\varphi $. If $\varphi $ is bounded, it is well known that $T_\varphi $ is bounded and selfadjoint on $H^2$. On the other hand, if $\varphi $ is unbounded, one can ask whether 
 $[T_{\varphi }]_{\rm max}= [T_{\varphi }]_{\rm min}$, where $[T_\varphi ]_{\rm min}$ is the closure of the operator $T_\varphi $ defined on the space of polynomials, while $[T_\varphi ]_{\rm max}$ has the domain 
$$\Dom ([T_\varphi]_{\rm max})=\{ f\in H^2: \mathbb P(\varphi f)\in L^2  \} $$
and $\mathbb P(\varphi f)$ is understood in the sense of distributions on $\bbT $. According to \cite{HaWi}, there exist unbounded real valued symbols $\varphi \in L^2(\bbT)$ such that $[T_{\varphi }]_{\rm max}\ne [T_{\varphi }]_{\rm min}$. For completeness, we recall this construction in Appendix~A.

For a general matrix $A=\{a_{jk}\}_{j,k=0}^\infty$, such that 
$$
\sum_{k=0}^\infty \abs{a_{jk}}^2<\infty \quad \forall j,
$$
one can define the maximal and minimal realisations $A_{\rm max}$ and $A_{\rm min}$ in the same spirit as above (see e.g. \cite[Example V.3.13]{Kato}) and ask whether $A_{\rm max}=A_{\rm min}$. Obviously, Jacobi matrices fit into this framework. Moreover, both Hankel and Toeplitz operators with symbols in $L^2(\bbT)$ fit in this framework as well since they can be realised as infinite matrices with respect to the standard basis $\{z^k\}_{k=0}^\infty$ of $H^2$. The above discussion demonstrates that the question of whether $A_{\rm max}=A_{\rm min}$ is extremely delicate and depends on the structure of the particular class of matrices $A$.

%%%%%%%%%%%%%%%%%%%%%%%%%%
\subsection{The flow of the cubic Szeg\H{o} equation}\label{sec.a2}
%%%%%%%%%%%%%%%%%%%%%%%%%%

The cubic Szeg\H{o} equation is the evolution equation on $H^2$
\begin{equation}
\label{a4}
i\partial_tu=\bbP(\abs{u}^2u)\ .
\end{equation}
It was introduced in \cite{GG00} where it was proved to be well-posed on the intersection of 
$H^2$ with the Sobolev space  $W^{s,2}(\bbT )$, for every $s\geq \frac 12$. 
The importance of this equation, in particular, stems from the fact that it displays long time transition of energy to high frequencies, expressed via the growth of high regularity Sobolev norms \cite[Theorem 1.0.1]{GG17}. 
More recently, the wellposedness was extended to $\BMOA(\bbT )$ in \cite{GK}.

The cubic Szeg\H{o} equation has been identified as a model equation for evolution PDEs  displaying some lack of dispersion. More precisely, in \cite{GGhalf} this equation was proved to be the totally resonant normal form associated to the cubic half-wave equation on the circle, 
$$i\partial_tu-|D|u=\vert u\vert^2u\ ,$$
providing some effective dynamics of the latter equation for some long time regime. Furthermore,  H.~Xu \cite{Xu} constructed modified scattering theory for the nonlinear Schr\"odinger half-wave equation on the cylinder $\bbR_x\times \bbT_y$,
$$i\partial_tu+\partial_x^2u-|D_y|u=|u|^2u\ ,$$
with small and decaying data, proving that the long time dynamics of this equation with such data is completely described by the cubic Szeg\H{o} equation in the  $y$ variable.
As a consequence, the growth of high regularity Sobolev norms is carried over to the nonlinear Schr\"odinger half-wave equation.

Let us come back to cubic Szeg\H{o} equation \eqref{a4} itself. It is easily seen to be the Hamiltonian equation with respect to the symplectic structure
$$
\omega(u,v)=\Im \jap{u,v}_{H^2}
$$
on $H^2$ and the Hamiltonian 
$$
E(u)=\frac14\norm{u}_{H^4}^4.
$$
In addition, the cubic Szeg\H{o} equation possesses a Lax pair structure, discovered in \cite{GG00} and completed in \cite{GG12}, which is fundamental to its analysis. This Lax pair structure and its consequences will be discussed in Section~\ref{sec.c}. Here we only say that the Lax operator is a Hankel operator with the symbol $u$; this gives the connection with the first part of the paper. 

The purpose of this paper is to extend the flow of the cubic Szeg\H{o} equation \eqref{a4} to the whole space $H^2$. Let $\Phi(t)$, $t>0$, be the (non-linear) flow map of \eqref{a4}, defined on $\BMOA(\bbT)$, i.e. if $u=u(t,z)$ is a solution to \eqref{a4}, then 
$$
\Phi(t)u(0,\cdot)=u(t,\cdot). 
$$
%%%%%%%%%%%%%%
\begin{theorem}\label{thm.a2}
%%%%%%%%%%%%%%
The flow map $\Phi(t)$ can be extended to $H^2$ such the following properties are satisfied:
\begin{enumerate}[\rm (i)]
\item
$\norm{\Phi(t)u}_{H^2}=\norm{u}_{H^2}$ for all $u\in H^2$;
\item
if $\norm{u_n-u}_{H^2}\to0$, then $\norm{\Phi(t)u_n-\Phi(t) u}_{H^2}\to0$ as $n\to\infty$;
\item
the map $t\mapsto\Phi(t)u$ is continuous in $H^2$ for all $u\in H^2$;
\item
if $u\in H^4$, then $\Phi(t)u\in H^4$ and $E(\Phi(t)u)=E(u)$. 
\end{enumerate}
\end{theorem}

Theorem~\ref{thm.a2} is sharp in the following sense: the flow $\Phi(t)$ cannot be continuously extended to   $H^2\cap W^{-\delta,2}(\bbT)$ with any $\delta>0$. The corresponding example is discussed in Section~\ref{sec.c}.

The statement of Theorem~\ref{thm.a2} may look  intriguing at first glance, since it claims the extension of  the flow map to square integrable functions, while the nonlinearity in the equation is cubic. The following  example shows that this can be done for very elementary equations. Indeed, the initial value problem
$$i\partial_tu(t,z)=|u(t,z)|^2u(t,z)\ ,\ u(0,z)=u_0(z)$$
can be solved as 
$$u(t,z)=u_0(z)\, {\rm e}^{-it|u_0(z)|^2}\ ,$$
and it is easy to check that the latter expression is well defined 
for $u_0\in L^2(\bbT)$ and the map $u_0\mapsto u(t,\cdot)$ is continuous on $L^2(\bbT)$.

As a low regularity sharp wellposedness result for an integrable Hamiltonian PDE, Theorem~\ref{thm.a2} can be compared to recent results on the cubic nonlinear Schr\"odinger equation \cite{HKV},  the Benjamin--Ono equation \cite{GKT}, and  the derivative nonlinear Schr\"odinger equation \cite{ HKNV}.
However, here the Lax operator is  of a different nature, and the methods for proving the theorem are specific to Hankel operators. 

The proof of Theorem~\ref{thm.a2} is based on the explicit formula for the flow $\Phi(t)$ on $\BMOA(\bbT)$ which was established in \cite{GG2}. We recall this formula in Section~\ref{sec.c}; it involves Hankel operators. We also use Theorem~\ref{thm.a1}, some tools from the Sz.-Nagy--Foias theory of contractions (Wold decomposition for isometries) and the Kato-Rosenblum theorem from scattering theory.

\subsection{The structure of the paper}
In Section~\ref{sec.b} we prove Theorem~\ref{thm.a1}. 
In Section~\ref{sec.c} we recall the Lax pair framework for the cubic Szeg\H{o} equation and reduce the proof of Theorem~\ref{thm.a2} to part (i) of the same Theorem. This part is proved in Section~\ref{sec.d}.

\subsection{Notation}\label{sec.a4}
We will mostly work in the Hilbert space $H^2$, and $\jap{f,g}$ (resp. $\norm{f}$) will refer to the inner product (resp. norm) in $H^2$. All other norms and inner products will be supplied with subscripts, e.g. $\norm{f}_{L^1(\bbT)}$.

The function identically equal to $1$ in $H^2$ is denoted by $\1$. 
We denote by $S$ the standard shift operator in $H^2$, acting as
$$
Sf(z)=zf(z), \quad f\in H^2.
$$
The backwards shift $S^*$ acts on $H^2$ by 
$$
S^*f(z)=\overline{z}(f(z)-\wh f_0). 
$$
We denote by $d\lambda(z)$ the normalised Lebesgue measure on $\bbT$.

\subsection{Acknowledgement}
The authors are grateful to S.~Treil for useful discussions. In particular, the proofs in Section~\ref{sec.d} are inspired by our recent joint work \cite{GPT}.
We are grateful to the referees for careful reading of the paper and for constructive remarks.

%%%%%%%%%%%%%%%%%%%%%%%%%%%%%%%%%%%%%%%%%%%%%%%%%
%%%%%%%%%%%%%%%%%%%%%%%%%%%%%%%%%%%%%%%%%%%%%%%%%
\section{Unbounded Hankel operators}\label{sec.b}
%%%%%%%%%%%%%%%%%%%%%%%%%%%%%%%%%%%%%%%%%%%%%%%%%
%%%%%%%%%%%%%%%%%%%%%%%%%%%%%%%%%%%%%%%%%%%%%%%%%

In this section we prove Theorem~\ref{thm.a1} and establish an auxiliary result (Theorem~\ref{thm.b2}) on the strong resolvent convergence of $\Gamma_u^*\Gamma_u$. 

\subsection{Proof of Theorem~\ref{thm.a1}(i): $\Gamma_u$ is closed}
This part is fairly standard. Suppose the sequence $f^{(k)}$ of elements of $\Dom \Gamma_u$ is Cauchy in the graph norm. The sequence $f^{(k)}$ is Cauchy in the usual norm of $H^2$ and therefore $\norm{f^{(k)}-f}\to0$ for some $f\in H^2$. We also have $\norm{\Gamma_uf^{(k)}-g}\to0$ for some $g\in H^2$; we need to prove that $f\in \Dom \Gamma_u$ and $\Gamma_uf=g$. 

Denote by $\bbP_N$ the projection from $L^1(\bbT)$ onto the set of all polynomials in $z$ of degree $\leq N$. Convergence $\norm{\Gamma_uf^{(k)}-g}\to0$ implies  that,  for all $N$,  we have 
\begin{equation}\label{z1} 
\norm{\bbP_N(\Gamma_uf^{(k)}-g)}\to 0\ . 
\end{equation}
On the other hand,  since $\norm{f^{(k)}-f}\to0$, we have $\norm{uJf^{(k)}-uJf}_{L^1(\bbT)}\to0$ and therefore all Fourier coefficients of $uJf^{(k)}$ converge to the corresponding Fourier coefficients of the $L^1$ function $uJf$. It follows that 
$$
\norm{\bbP_N(\Gamma_uf^{(k)})-\bbP_N(uJf)}\to0, \quad k\to\infty,
$$
for any finite $N$. Comparing with \eqref{z1}, we obtain
\begin{equation}\label{z2}
{\bbP_N(uJf)}=\bbP_N g
\end{equation}
for all $N$, and therefore $\bbP(uJf)\in H^2$, i.e. $f\in \Dom \Gamma_u$. 
Thus, \eqref{z2} can be rewritten as
$\bbP_N(\Gamma_uf)=\bbP_Ng.$
As $N$ is arbitrary, we conclude that $\Gamma_uf=g$. The proof of Theorem~\ref{thm.a1}(i) is complete.
 
\subsection{The commutation relation and the graph of $\Gamma_u$}
For any $f\in L^2(\bbT)$, we have
$$
\bbP(\overline{z}f)=S^*\bbP f;
$$
this implies the commutation relation 
\begin{equation}
\Gamma_u S=S^*\Gamma_u. 
\label{b1}
\end{equation}
It should be more precisely stated as follows: for any $f\in \Dom \Gamma_u$, we have $Sf\in \Dom \Gamma_u$ and $\Gamma_u Sf=S^*\Gamma_u f$. 

The commutation relation \eqref{b1} is fundamental for the theory of Hankel operators. In fact, it can be proved that if a bounded operator $\Gamma$ satisfies the commutation relation \eqref{b1} with the shift operator, then $\Gamma$ is necessarily a Hankel operator, see e.g. \cite[Section 1.1]{Peller}.

The importance of the commutation relation \eqref{b1} in the context of the proof of Theorem~\ref{thm.a1} is as follows. It implies that the graph of $\Gamma_u$ in $H^2\oplus H^2$, i.e. 
$$
\Graph(\Gamma_u)=\{(f,\Gamma_u f): f\in \Dom \Gamma_u\},
$$
satisfies the invariance property 
$$
(S\oplus S^*)\Graph(\Gamma_u)\subset \Graph(\Gamma_u). 
$$
We need a description of subspaces of $H^2\oplus H^2$ satisfying this property.

\subsection{Invariant subspaces of $S\oplus S^*$}
The description of invariant subspaces of $S\oplus S^*$
can be obtained from the Sz.-Nagy--Foias theory of contractions; the details have been worked out in \cite{Timotin}. 
%%%%%%%%%%%
\begin{theorem}
\cite[Theorem 4.1]{Timotin}
\label{thm.b1}
%%%%%%%%%%%
The invariant subspaces $Y$ of $S\oplus S^*$ acting on $H^2\oplus H^2$ are the following:
\begin{enumerate}[\rm (i)]
\item
Splitting invariant subspaces; that is, $Y=X\oplus X'$ with $X$ invariant to $S$ and $X'$ invariant to $S^*$. 
\item
Nonsplitting invariant subspaces. These are of the form
\begin{equation}
Y=\{\bigl(\theta_{21} w_1+\theta_{22}w_2, \bbP(\overline{z}\theta_{11} Jw_1)+\bbP(\overline{z}\theta_{12} Jw_2)\bigr): w_1,w_2\in H^2\},
\label{b2}
\end{equation}
where $\theta_{ij}$ are functions in the unit ball of $H^\infty$, such that $\theta_{11}$ and $\theta_{12}$ are not proportional, and 
\begin{equation}
\abs{\theta_{11}}^2+\abs{\theta_{12}}^2=1, 
\quad
\abs{\theta_{21}}^2+\abs{\theta_{22}}^2=1, 
\quad
\theta_{11}^\sharp \theta_{21}+\theta_{12}^\sharp\theta_{22}=0.
\label{b2a}
\end{equation}
\end{enumerate}
\end{theorem}

In fact, \cite[Theorem 4.1]{Timotin} gives the description of the invariant subspaces of $S\oplus S_*$ in $H^2\oplus H^2_-$, where $H^2_-=L^2(\bbT)\cap (H^2)^\perp$ and $S_*$ is the compression of the operator of multiplication by $z$ onto $H^2_-$. However, the operators $S^*$ on $H^2$ and $S_*$ on $H^2_-$ are related through the simple unitary equivalence $f\mapsto \overline{z}f(\overline{z})$; applying this unitary equivalence yields the version of the theorem that we give above.

It is clear that a splitting subspace $X\oplus X'$ with $X'\not=\{0\}$ cannot be a graph of an operator in $H^2\oplus H^2$, because it contains vectors of the form $(0,f)$ with $f\not=0$. Further, $X\oplus\{0\}$ can only be a graph of a zero operator. Thus, part (i) of Theorem~\ref{thm.b1} is of no relevance to us. 

Let us rewrite the second component in \eqref{b2} in a more compact form. We have 
$$
\bbP(\overline{z}\theta_{1j}Jw_j)=\bbP((S^*\theta_{1j})Jw_j)=\Gamma_{S^*\theta_{1j}}w_j 
$$
for $j=1,2$. 
Denoting for brevity
$$
\varphi_{1}=S^*\theta_{11}, \quad \varphi_{2}=S^*\theta_{12}, 
$$
we can therefore rewrite \eqref{b2} in the more compact form 
$$
Y=\{\bigl(\theta_{21} w_1+\theta_{22}w_2,\Gamma_{\varphi_{1}}w_1+\Gamma_{\varphi_{2}}w_2\bigr): w_1,w_2\in H^2\}.
$$
Conditions \eqref{b2a} on $\theta_{ij}$ will play no role in our proof below; all that matters to us is that $\theta_{ij}$ (and therefore $\varphi_{j}$) are elements of $H^\infty$. 

Finally, it will be important for us to characterise the orthogonal complements of subspaces $Y$. Since $Y$ is invariant under $S\oplus S^*$, its orthogonal complement $Y^\perp$ is invariant under the adjoint $S^*\oplus S$. Thus, the description of $Y^\perp$ can be obtained by the interchange of coordinates in Theorem~\ref{thm.b1}. 

\subsection{Proof of Theorem~\ref{thm.a1}(ii): polynomials are dense in $\Dom \Gamma_u$}
Let $\calP\subset H^2$ be the set of all polynomials of $z$. Let $Y_0$, $Y$ be the linear subsets of $H^2\oplus H^2$, 
$$
Y_0=\{(p,\Gamma_u p): p\in \calP\}, 
\quad 
Y=\Graph(\Gamma_u)=\{(f,\Gamma_u f): f\in H^2, \Gamma_u f\in H^2\}.
$$
Obviously, $Y_0\subset Y$; we have already proved that $Y$ is closed and we need to prove that  the closure of $Y_0$ coincides with $Y$. We will do this by proving that $Y_0^\perp\subset Y^\perp$, where the orthogonal complements are taken in $H^2\oplus H^2$. 

Let us first describe $Y_0^\perp$. Since $Y_0$ is invariant for $S\oplus S^*$, its orthogonal complement $Y_0^\perp$ is invariant for the adjoint operator $S^*\oplus S$. The description of the invariant subspaces of $S^*\oplus S$ is again given by Theorem~\ref{thm.b1} with interchanging the two components in $H^2\oplus H^2$. Since $Y_0$ is non-splitting, its orthogonal complement is also non-splitting. We conclude that $Y_0^\perp$ is of the form
\begin{equation}
Y_0^\perp=\{\bigl(\Gamma_{\varphi_1}w_1+\Gamma_{\varphi_2}w_2,\theta_1w_1+\theta_2w_2\bigr): w_1,w_2\in H^2\}
\label{b2b}
\end{equation}
with some parameters $\theta_1,\theta_2,\varphi_1,\varphi_2\in H^\infty$. 
The condition of orthogonality to $Y_0$ means that for all polynomials $p$,
\begin{equation}
\jap{\Gamma_{\varphi_1}w_1+\Gamma_{\varphi_2}w_2,p}
+
\jap{\theta_1w_1+\theta_2w_2,\Gamma_u p}
=0.
\label{b3a}
\end{equation}
Let us rewrite the second term here in a different way. Since $p^\sharp=\overline{Jp}$, we have 
$$
\jap{\theta_j w_j,\Gamma_u p}
=
\jap{\theta_j w_j,\bbP(uJp)}
=
\jap{\theta_j w_j,uJp}
=
\jap{\overline{Jp}\theta_jw_j,u}
=
\jap{p^\sharp\theta_j  w_j,u}
$$
for $j=1,2$,
and therefore condition \eqref{b3a} rewrites as 
\begin{equation}
\jap{\Gamma_{\varphi_1}w_1+\Gamma_{\varphi_2}w_2,p}
+
\jap{p^\sharp\theta_1 w_1+p^\sharp\theta_2 w_2,u}
=0
\label{b4}
\end{equation}
for all polynomials $p$. 

Now let us prove that $Y_0^\perp\subset Y^\perp$. Fix any $f\in \Dom \Gamma_u$; we need to check that $(f,\Gamma_uf)\perp Y_0^\perp$. It is clear that taking $w_1,w_2\in\calP$ in \eqref{b2b}, we obtain a dense subset of $Y_0^\perp$. Thus, it suffices to check that 
\begin{equation}
\jap{\Gamma_{\varphi_1}w_1+\Gamma_{\varphi_2}w_2,f}
+
\jap{\theta_1w_1+\theta_2w_2,\Gamma_uf}
=0
\label{b5}
\end{equation}
for all $w_1,w_2\in \calP$. Consider the second term in \eqref{b5}:
$$
\jap{\theta_jw_j,\bbP(uJf)}
=
\int_{\bbT} \theta_j(z)w_j(z)\overline{u(z)}\overline{f(\overline{z})}d\lambda(z),
$$
where $d\lambda$ is the normalised Lebesgue measure on $\bbT$;
the integral here is well-defined, since $\theta_j,w_j\in H^\infty$ and $u,f\in H^2$. We can rewrite the last integral as
$$
\jap{f^\sharp\theta_j  w_j,u},
$$
where the inner product is well-defined as both functions $ f^\sharp \theta_jw_j$ and  $u$ are in $H^2$. We conclude that condition \eqref{b5} that is to be checked is equivalent to 
\begin{equation}
\jap{\Gamma_{\varphi_1}w_1+\Gamma_{\varphi_2}w_2,f}
+
\jap{f^\sharp\theta_1 w_1+f^\sharp\theta_2 w_2,u}
=0,
\label{b6}
\end{equation}
for all $w_1,w_2\in\calP$. 

Now let $\{p_n\}$ be a sequence of polynomials such that $\norm{p_n-f}\to0$ as $n\to\infty$. For each $p=p_n$ and each $w_1,w_2\in\calP$, orthogonality condition \eqref{b4} holds true. We can now pass to the limit $n\to\infty$ in \eqref{b4}, because 
$$
\norm{p_n^\sharp\theta_j  w_j-f^\sharp\theta_j w_j}\to0
$$
as $n\to\infty$ for $j=1,2$. This yields \eqref{b6}. We have checked that $Y_0^\perp\subset Y^\perp$. 
The proof of Theorem~\ref{thm.a1}(ii) is complete.

\subsection{Proof of Theorem~\ref{thm.a1}(iii): $\Gamma_u^*=\Gamma_{u^\sharp}$}

Suppose $f\in\Dom\Gamma_u^*$, i.e. for some $h\in H^2$ we have
\begin{equation}
\jap{\Gamma_u g,f}=\jap{g,h}, \quad\forall g\in\Dom\Gamma_u
\label{b7}
\end{equation}
(in this case $h=\Gamma_u^*g$). 
In particular, this relation is true for polynomials $g$. For polynomials $g$ we have
\begin{align*}
\jap{\Gamma_u g,f}&=\jap{\bbP(uJg),f}=\jap{uJg,f}_{L^2(\bbT)}
\\
&=\int_{\bbT} u(z)g(\overline{z})\overline{f(z)}d\lambda(z)
=
\int_{\bbT} u(\overline{z})g(z)\overline{f(\overline{z})}d\lambda(z)
\\
&=
\int_{\bbT} \overline{u^\sharp(z)}g(z) \overline{Jf(z)}d\lambda(z)
=
\int_{\bbT}g(z)\overline{h_*(z)}d\lambda(z), 
\end{align*}
where $h_*=u^\sharp Jf\in L^1(\bbT)$. 
Comparing with \eqref{b7}, we find that $\bbP h_*=h\in H^2$; thus, $f\in \Dom \Gamma_{u^\sharp}$ and $\Gamma_{u^\sharp}f=h$. 

We have proved that $\Dom\Gamma_u^*\subset \Dom\Gamma_{u^\sharp}$ and the operators $\Gamma_u^*$ and $\Gamma_{u^\sharp}$ coincide on $\Dom\Gamma_u^*$, i.e. $\Gamma_u^*\subset \Gamma_{u^\sharp}$. 

Conversely, suppose $f\in\Gamma_{u^\sharp}$, i.e. $\bbP(u^\sharp Jf)=h\in H^2$. Then, by exactly the same argument, we find that \eqref{b7} holds true for all polynomials $g$. By part (ii) of the theorem, polynomials are dense in $\Dom \Gamma_u$, and so we can extend \eqref{b7} by continuity to all $g\in\Dom\Gamma_u$. This shows that $\Dom\Gamma_{u^\sharp}\subset\Dom\Gamma_u^*$. Summarizing, we obtain $\Gamma_u^*=\Gamma_{u^\sharp}$ and the proof of Theorem~\ref{thm.a1} is complete. \qed

\subsection{The strong resolvent convergence of $\Gamma_u^*\Gamma_u$ and $\Gamma_u\Gamma_u^*$}
Let $u\in H^2$ and let $\Gamma_u$ be as in Theorem~\ref{thm.a1}. Consider the closed operator $\Gamma_u^*\Gamma_u$ with the domain
$$
\Dom \Gamma_u^*\Gamma_u=\{f\in\Dom\Gamma_u: \Gamma_uf\in\Dom \Gamma_u^*\}.
$$
This is the self-adjoint operator, corresponding to the closed bilinear form
$$
\jap{\Gamma_uf,\Gamma_ug}, \quad f,g\in \Dom \Gamma_u \ , 
$$
see e.g. \cite[Section VIII.6]{RS1}. 
In the same way one defines $\Gamma_u\Gamma_u^*$. 

We recall that a sequence of self-adjoint operators $A_n$ is said to converge to $A$ in the \emph{strong resolvent sense}, if $(A_n-\lambda)^{-1}\to(A-\lambda)^{-1}$ in the strong operator topology for all $\lambda$ with $\Im \lambda\not=0$. 
%%%%%%%%%%%%%%%%%
\begin{theorem}\label{thm.b2}
%%%%%%%%%%%%%%%%%
Let $\{u_n\}_{n=1}^\infty$ be a sequence of elements of $H^2$ with $\norm{u_n-u}\to0$ as $n\to\infty$. Then 
$$
\Gamma_{u_n}^*\Gamma_{u_n}\to \Gamma_u^*\Gamma_u
\quad \text{ and }\quad
\Gamma_{u_n}\Gamma_{u_n}^*\to \Gamma_u\Gamma_u^*
$$
in the strong resolvent sense. 
\end{theorem}
\begin{proof}
We prove the convergence of $\Gamma_{u_n}^*\Gamma_{u_n}$; from here, replacing $u_n$ by $u_n^\sharp$, one obtains the convergence of $\Gamma_{u_n}\Gamma_{u_n}^*$ as well. 
Fix $\lambda$ with $\Im \lambda\not=0$ and $f\in H^2$; denote 
$$
\varphi_n=(\Gamma_{u_n}^*\Gamma_{u_n}-\lambda)^{-1}f, 
\quad
\varphi=(\Gamma_u^*\Gamma_u-\lambda)^{-1}f.
$$
We need to prove that $\norm{\varphi_n-\varphi}\to0$ as $n\to\infty$. 
A simple calculation with the first resolvent identity (see e.g. \cite[Problem VIII.20]{RS1}) shows that it is sufficient to prove the weak convergence $\varphi_n\to\varphi$.

Since $\Im \lambda\not=0$, we have
$$
\norm{\varphi_n}\leq C_1(\lambda)\norm{f}. 
$$
Next, 
$$
\norm{\Gamma_{u_n}\varphi_n}^2-\lambda\norm{\varphi_n}^2
=
\jap{\Gamma_{u_n}^*\Gamma_{u_n}\varphi_n,\varphi_n}-\lambda\norm{\varphi_n}^2
=
\jap{f,\varphi_n}=\jap{f,(\Gamma_{u_n}^*\Gamma_{u_n}-\lambda)^{-1}f}, 
$$
and so 
$$
\norm{\Gamma_{u_n}\varphi_n}^2
\leq 
\abs{\jap{f,(\Gamma_{u_n}^*\Gamma_{u_n}-\lambda)^{-1}f}}+\abs{\lambda}\norm{\varphi_n}^2
\leq
C_2(\lambda)\norm{f}^2.
$$
Thus, we can select a subsequence such that $\varphi_n\to \widetilde\varphi$ weakly and $\Gamma_{u_n}\varphi_n\to\psi$ weakly for some $\widetilde \varphi,\psi\in H^2$. 

Let us prove that $\widetilde\varphi\in\Dom\Gamma_u$ and $\Gamma_u\widetilde\varphi=\psi$. For any polynomial $p$, we have 
$$
\jap{\varphi_n,\Gamma_{u_n}^*p}
=
\jap{\Gamma_{u_n}\varphi_n,p}
\to 
\jap{\psi,p}.
$$
Since $p\in H^\infty$, we also have $\norm{\Gamma_{u_n}^*p-\Gamma_u^*p}\to0$, and so 
$$
\jap{\varphi_n,\Gamma_{u_n}^*p}
\to
\jap{\widetilde\varphi,\Gamma_u^*p}.
$$
It follows that 
$$
\jap{\widetilde\varphi,\Gamma_u^*p}=\jap{\psi,p}
$$
for any polynomial $p$. By Theorem~\ref{thm.a1}, polynomials are dense in $\Dom\Gamma_u^*$ with respect to the graph norm and therefore the last identity extends by continuity to all $p\in\Dom\Gamma_u^*$. It follows that $\widetilde\varphi\in\Dom\Gamma_u$ and 
$\Gamma_u\widetilde\varphi=\psi$. 

By the definition of $\varphi_n$, for any polynomial $p$ we have
$$
\jap{\Gamma_{u_n}\varphi_n,\Gamma_{u_n}p}
-
\lambda\jap{\varphi_n,p}
=
\jap{f,p}.
$$
By the previous step of the proof, we have $\varphi_n\to\widetilde\varphi$ and $\Gamma_{u_n}\varphi_n\to\Gamma_u\widetilde\varphi$ weakly. We also have $\norm{\Gamma_{u_n}p-\Gamma_up}\to0$. Thus, passing to the limit in the last identity, we find
$$
\jap{\Gamma_u\widetilde\varphi,\Gamma_up}-\lambda\jap{\widetilde\varphi,p}
=
\jap{f,p}
$$
for all polynomials $p$. Using Theorem~\ref{thm.a1} again, we extend this identity from polynomials $p$ to all elements $p\in\Dom\Gamma_u$. It follows that $\Gamma_u\widetilde\varphi\in\Dom\Gamma_u^*$ and 
$$
\Gamma_u^*\Gamma_u\widetilde\varphi=\lambda\widetilde\varphi+f,
$$
which can be equivalently rewritten as 
$$
\widetilde\varphi=(\Gamma_u^*\Gamma_u-\lambda)^{-1}f.
$$
Now recall that the element in the right hand side here is $\varphi$. 
We have proved that $\varphi_n\to\widetilde\varphi=\varphi$ weakly
over a subsequence. 
Finally, in order to prove that the weak convergence $\varphi_n\to\varphi$ holds over the whole sequence, we use the standard trick: assume that the convergence fails over some subsequence, then use the same argument to select a convergent subsubsequence -- contradiction.
The proof of Theorem~\ref{thm.b2} is complete.
\end{proof}

%%%%%%%%%%%%%%%%%%%%%%%%%%%%%%%%%%%%%%%%%%%%%%%%%
%%%%%%%%%%%%%%%%%%%%%%%%%%%%%%%%%%%%%%%%%%%%%%%%%
\section{The flow of the cubic Szeg\H{o} equation}\label{sec.c}
%%%%%%%%%%%%%%%%%%%%%%%%%%%%%%%%%%%%%%%%%%%%%%%%%
%%%%%%%%%%%%%%%%%%%%%%%%%%%%%%%%%%%%%%%%%%%%%%%%%

\subsection{Anti-linear Hankel operators}
The Lax pair formalism for the cubic Szeg\H{o} equation is  described in terms of the anti-linear variant of the definition of Hankel operators, viz. 
$$
H_uf=\Gamma_uf^\sharp=\bbP(u\overline{f}).
$$
We first assume $u\in H^\infty$ and discuss the relevant algebraic aspects of this definition. The operator $H_u$ satisfies the symmetry relation 
$$
\jap{H_u f,g}=\jap{H_ug,f}.
$$
Indeed, by the definition of $H_u$, both sides are equal to 
$$
\int_\bbT u(z)\overline{f(z)}\overline{g(z)}d\lambda(z).
$$
Since $S^*(f^\sharp)=(S^*f)^\sharp$, the commutation relation \eqref{b1} also holds for $H_u$: 
$$
S^*H_u=H_uS.
$$
By a direct calculation, the operator $H_u^2$ is linear and 
$$
H_u^2=\Gamma_u\Gamma_u^*.
$$

Now let $u\in H^2$, so that $H_u$ may be unbounded. Since $\norm{f^\sharp}=\norm{f}$, all the relevant properties of $H_u$ can be directly inferred from those of $\Gamma_u$. For future reference, below we record these properties.

For $u\in H^2$ we set
\begin{equation}
\Dom H_u=\{f\in H^2: \bbP(u\overline{f})\in H^2\}
\label{c6}
\end{equation}
and define 
\begin{equation}
H_uf=\bbP(u\overline{f}), \quad f\in \Dom H_u.
\label{c7}
\end{equation}
As in the linear case, the graph norm of $H_u$ is defined as
$$
\norm{f}_{H_u}=(\norm{f}^2+\norm{H_u f}^2)^{1/2}, \quad f\in\Dom H_u.
$$
Theorem~\ref{thm.a1} translates into the language of anti-linear Hankel operators as follows:
%%%%%%%%%%%%%%%%%
\begin{theorem}\label{thm.c1}
%%%%%%%%%%%%%%%%%
Let $u\in H^2$, and let $H_u$ be the anti-linear Hankel operator defined by \eqref{c7} on the domain \eqref{c6}. Then: 
\begin{enumerate}[\rm (i)]
\item
$H_u$ is closed, i.e. $\Dom H_u$ is closed with respect to the graph norm of $H_u$;
\item
the set of all polynomials is dense in $\Dom H_u$ with respect to the graph norm;
\item
suppose for some $f,h\in H^2$  we have 
$$
\jap{H_ug,f}=\jap{h,g}, \quad \forall g\in\Dom H_u.
$$
Then $f\in\Dom H_u$ and $H_uf=h$. 
\end{enumerate}
\end{theorem}
\begin{proof}
(i) and (ii) immediately follow from Theorem~\ref{thm.a1}; only (iii) requires proof. Translating the hypothesis of (iii) into the language of linear Hankel operators, we obtain the following statement:
$$
\jap{\Gamma_ug^\sharp,f}=\jap{h,g}, \quad \forall g^\sharp\in\Dom\Gamma_u.
$$
Since $\jap{h,g}=\jap{g^\sharp,h^\sharp}$, we can rewrite this as 
$$
\jap{\Gamma_ug^\sharp,f}=\jap{g^\sharp,h^\sharp}, \quad \forall g^\sharp\in\Dom\Gamma_u.
$$
This means that $f\in \Dom\Gamma_u^*$ and $\Gamma_u^*f=h^\sharp$. 
By Theorem~\ref{thm.a1}(iii), we find that $f\in\Dom\Gamma_{u^\sharp}$ and $\Gamma_{u^\sharp}f=h^\sharp$. Applying $\sharp$ to both sides here, we obtain $f^\sharp\in\Dom \Gamma_u$ and $\Gamma_uf^\sharp=h$, or equivalently $f\in\Dom H_u$ and $H_uf=h$, as claimed.
\end{proof}

\begin{corollary}\label{cr.c2}
Let $u\in H^2$; then for all $f,g\in\Dom H_u$, we have 
$$
\jap{H_uf,g}=\jap{H_ug,f}\ .
$$
\end{corollary}
\begin{proof} For polynomials $f,g$ this is a direct calculation; now use Theorem~\ref{thm.c1}(ii) to extend to all $f,g\in\Dom H_u$.
\end{proof}

Finally, we note that when $H_u$ is unbounded, the precise form of the commutation relation $S^*H_u=H_uS$ is as follows: for any $f\in \Dom H_u$, one has $Sf\in \Dom H_u$ and $S^*H_uf=H_uSf$.

%%%%%%%%%%%%%%%%%%%%%%%%%
\subsection{The operator $\widetilde H_u$}
%%%%%%%%%%%%%%%%%%%%%%%%%
The Lax pair formalism of the cubic Szeg\H{o} equation involves two Hankel operators: $H_u$ and $H_{S^*u}$. For typographical reasons, we denote 
$$
\widetilde H_u:=H_{S^*u}.
$$
First let us assume that $u\in H^\infty$ and discuss the algebraic properties of $\wt H_u$. 
Observe that $\widetilde H_u$ can be alternatively defined by 
\begin{equation}
\widetilde H_u=S^*H_u=H_uS.
\label{c4}
\end{equation}
Recall that $S$ and $S^*$ satisfy the identities
\begin{equation}
S^*S=I, \quad SS^*=I-\jap{\cdot,\1}\1,
\label{c1}
\end{equation}
where $\jap{\cdot,\1}\1$ is the (rank one) operator of projection onto constants. 
As a consequence of this, the operators $H_u$ and $\wt H_u$ satisfy a rank one identity which is fundamental to the theory. We have
$$
\wt H_u^2=H_uSS^*H_u=H_u(I-\jap{\cdot,\1}\1)H_u=H_u^2-\jap{\cdot,H_u\1}H_u\1.
$$
By the definition of $H_u$, we have $H_u\1=u$, and so we finally obtain 
\begin{equation}
\wt H_u^2=H_u^2-\jap{\cdot,u}u.
\label{c2}
\end{equation}

Now let $u\in H^2$, so that $H_u$ may be unbounded. Then \eqref{c4} still holds true on $\Dom H_u$ and shows that $\Dom \wt H_u=\Dom H_u$. 
Identity \eqref{c2} should be understood in the quadratic form sense, i.e. 
$$
\norm{\wt H_u f}^2=\norm{H_u f}^2-\abs{\jap{f,u}}^2, \quad f\in \Dom H_u. 
$$

\subsection{Lax pair identities and formula for the flow}
Consider the cubic Szeg\H{o} equation \eqref{a4} for smooth initial data, say in $C^\infty (\bbT)\cap H^2$. In this case, as established in \cite{GG2}, the solution $u$ to the cubic Szeg\H{o} equation  satisfies the following two Lax pair identities:
\begin{align*}
\frac{dH_u}{dt}&=[B_u,H_u]\ ,\ B_u:=\frac i2H_u^2-iT_{\vert u\vert ^2}\ ,
\\
\frac{d\wt H_u}{dt}&=[\wt B_u,\wt H_u]\ ,\ \wt B_u:=\frac i2\wt H_u^2-iT_{\vert u\vert ^2}\ ,
\end{align*}
where $T_{\abs{u}^2}$ is the Toeplitz operator with the symbol $\abs{u}^2$.
\begin{remark*}
Observe that these Lax pair identities are formulated in terms of the anti-linear variant of Hankel operators $H_u$. We are not aware of an equivalent formulation in terms of the linear Hankel operators $\Gamma_u$. Thus, the Lax pair structure forces us to work with $H_u$ rather than $\Gamma_u$. 
\end{remark*}

As a consequence of these Lax pair identities, an explicit formula for the flow $\Phi(t)$ has been derived in \cite{GG2} for smooth initial data. The most convenient way to display this formula is to regard $\Phi(t)u\in H^2$ as a function of the complex variable $\abs{z}<1$ rather than as a function on the unit circle. With this convention, the formula from \cite{GG2} is 
\begin{equation}
(\Phi(t)u)(z)=\jap{(I-ze^{-itH_u^2}e^{it\wt H_u^2}S^*)^{-1}e^{-itH_u^2}u,\1}, \quad \abs{z}<1
\label{c3}
\end{equation}
for $u\in C^\infty (\bbT)\cap H^2$. 
Clearly, the operator norm of $e^{-itH_u^2}e^{it\wt H_u^2}S^*$ equals one, and therefore $(\Phi(t)u)(z)$ is well-defined and holomorphic for $\abs{z}<1$. 
The main ingredient of our proof is the following statement.
%%%%%%%%%%%%%%%%
\begin{theorem}\label{thm.c2}
%%%%%%%%%%%%%%%%
For any $u\in H^2$ and $t\in\bbR$, let $\Phi(t)u$ be the holomorphic function in the open unit disk, defined by \eqref{c3}. Then $\Phi(t)u\in H^2$ and 
$$
\norm{\Phi(t)u}=\norm{u}. 
$$
\end{theorem}
Obviously, this proves Theorem~\ref{thm.a2}(i). 
We give the proof of Theorem~\ref{thm.c2} in Section~\ref{sec.d}. In the rest of this section we prove the remaining statements of Theorem~\ref{thm.a2} and demonstrate its sharpness. 

\subsection{Proof of Theorem~\ref{thm.a2}(ii)}
Let $\norm{u_n-u}\to0$ in $H^2$; then we also have $\norm{u_n^\sharp-u^\sharp}\to0$. By Theorem~\ref{thm.b2} we have the strong resolvent convergence 
$$
\Gamma_{u_n}\Gamma_{u_n}^*\to\Gamma_u\Gamma_u^*
$$
or equivalently, the strong resolvent convergence 
$$
H_{u_n}^2\to H_u^2. 
$$ 
This implies \cite[Theorem VIII.20]{RS1} that $f(H_{u_n}^2)\to f(H_u^2)$ strongly for all bounded continuous functions $f$. In particular, 
$$
\exp(-itH_{u_n}^2)\to\exp(-itH_{u}^2)
$$
in the strong operator topology as $n\to\infty$. Similarly (since $\norm{S^*u_n-S^*u}\to0$), we have
$$
\exp(-it\wt H_{u_n}^2)\to\exp(-it\wt H_{u}^2)
$$
in the strong operator topology. It follows that $(\Phi(t)u_n)(z)\to(\Phi(t)u)(z)$  for any $\abs{z}<1$ and therefore $\Phi(t)u_n\to\Phi(t)u$ weakly in $H^2$. 
On the other hand, by Theorem~\ref{thm.c2}, we have 
$$
\norm{\Phi(t)u_n}=\norm{u_n}\to\norm{u}=\norm{\Phi(t)u},
$$
and so we conclude that $\norm{\Phi(t)u_n-\Phi(t)u}\to0$ as $n\to\infty$. 
\qed

\subsection{Proof of Theorem~\ref{thm.a2}(iii)}
Suppose $t_n\to t$ as $n\to\infty$. As in the proof of part (ii) of the theorem, using the strong convergence argument we find that $(\Phi(t_n)u)(z)\to(\Phi(t)u)(z)$  for any $\abs{z}<1$ and therefore $\Phi(t_n)u\to\Phi(t)u$ weakly in $H^2$. On the other hand, by Theorem~\ref{thm.c2},
$$
\norm{\Phi(t_n)u}=\norm{\Phi(t)u},
$$
and so  $\norm{\Phi(t_n)u-\Phi(t)u}\to0$ as $n\to\infty$. 
\qed

\subsection{Proof of Theorem~\ref{thm.a2}(iv)}

For every $u\in H^2$ and for every nonnegative real number $x$, let us consider
$$J(x,u):=\jap{(I+xH_u^2)^{-1}\1,\1}\ .$$
The function $x\mapsto J(x,u)$ is $C^\infty$ on $(0,\infty)$ with 
$$
\partial_x J(x,u)=-\jap{(I+xH_u^2)^{-2}H_u^2\1,\1}=-\norm{H_u(I+xH_u^2)^{-1}\1}^2.
$$
From here we see that this function is also $C^1$ up to $x=0$ with
$$\partial_xJ(0_+,u)=-\norm{H_u\1}^2=-\norm{u}^2.$$
Furthermore, for $x>0$,
\begin{align*}
\partial_x^2J(x,u)&=2\jap{H_u^4(I+xH_u^2)^{-3}\1,\1}
=
2\norm{H_u(I+xH_u^2)^{-3/2}H_u\1}^2
\\
&=
2\norm{H_u(I+xH_u^2)^{-3/2}u}^2.
\end{align*}
Observe that $u\in\Dom H_u$ if and only if 
$$
\sup_{x>0}\norm{H_u(I+xH_u^2)^{-3/2}u}^2<\infty,
$$
i.e. if and only if the second derivative $\partial_x^2J(x,u)$ remains bounded as $x\to0_+$, and in this case we have
$$
\partial_x^2J(0_+,u)=2\norm{H_uu}^2.
$$
On the other hand, by the definition of $\Dom H_u$, we have that $u\in\Dom H_u$ if and only if $\bbP(|u|^2)\in H^2$. Computing the norm  in terms of the Fourier coefficients, we find
$$
2\norm{\bbP(|u|^2)}^2
=2\sum_{n=0}^\infty \abs{\jap{\abs{u}^2,z^n}}^2
=\sum_{n=-\infty}^\infty \abs{\jap{\abs{u}^2,z^n}}^2
+\abs{\jap{\abs{u}^2,\1}}^2
=\norm{u}_{H^4}^4+\norm{u}^4
$$
and so $u\in\Dom H_u$ if and only if $u\in H^4$, and in this case
$$
2\norm{H_uu}^2=\norm{u}_{H^4}^4+\norm{u}^4.
$$
We conclude that $u\in H^4$ if and only if the second derivative $\partial_x^2J(x,u)$ remains bounded as $x\to0_+$, and in this case
\begin{equation}
\norm{u}_{H^4}^4
=
\partial_x^2J(0_+,u)-(\partial_xJ(0_+,u))^2.
\label{c9}
\end{equation}
\\
Next, if $u$ is smooth, it is known from the Lax pair structure that $J(x,u)$ is a conservation law of the cubic Szeg\H{o} equation (see e.g. \cite[Corollary~3]{GG00}). Moreover, we already noticed that, if $u_n\to u$ in $H^2$, then by Theorem~\ref{thm.b2} we have $H_{u_n}^2\to H_u^2$ in the sense of the strong convergence of resolvents. This means that $u\mapsto J(x,u)$ is a continuous functional on $H^2$. In view  of Theorem~\ref{thm.a2}(ii), we infer that, for every $u\in H^2$, 
$$J(x,\Phi (t)u)=J(x,u)\ .$$
Now take $u\in H^4$; then from \eqref{c9} we conclude that $\Phi (t)u\in H^4$ and that
$$\norm{\Phi(t)u}_{H^4}=\norm{u}_{H^4}\ .$$
The proof of Theorem~\ref{thm.a2}(iv) is complete. \qed

\subsection{Sharpness of Theorem~\ref{thm.a2}}

%%%%%%%%%%%%%%%%%%%%
\begin{proposition}\label{prp.a1}
%%%%%%%%%%%%%%%%%%%%
For every $\delta >0$, there exists a sequence $(v^\eps )$ of smooth (in fact, rational) solutions of the cubic Szeg\H{o} equation on the circle and a sequence $(t^\eps )$ of positive times such that, as $\eps \to 0$, 
$$t^\eps \to 0\ ,\ \| v^\eps (0)\| _{W^{-\delta ,2}}\to 0\ ,\ |\jap{v^\eps (t^\eps ),\1}|\to +\infty \ .$$
\end{proposition}
\begin{remark*}
Since $|\jap{v,\1}|\leq \| v\| _{W^{-\delta ,2}}$, this establishes immediate norm inflation in $W^{-\delta ,2}(\bbT)$ for the cubic Szeg\H{o} evolution. A slight modification of the argument below shows that, for every $t_0>0$, there exists a sequence $(v^\eps )$ of smooth solutions such that 
$$ \| v^\eps (0)\| _{W^{-\delta ,2}}\to 0\ ,\ |\jap{v^\eps (t_0 ),\1}|\to +\infty \ .$$
In particular, the  cubic Szeg\H{o}  flow map $\Phi(t_0)$ cannot be continuously extended to a map on $W^{-\delta ,2}(\bbT)$, nor even to a map from $W^{-\delta ,2}(\bbT)$ to the space of distributions $\mathcal D'(\bbT)$.  
\end{remark*} 

\begin{proof}
First of all we observe the following two scaling properties of the cubic Szeg\H{o} equation \eqref{a4}. If $u=u(t,z)$ is a solution of \eqref{a4}, then for every $R>0$ and for every positive integer $N$,
$$u_{R,N}(t,z):=Ru(R^2t,z^N)$$
is also a solution of the cubic Szeg\H{o} equation. Next, we refer to Section 4 of \cite{GG2} where the solution $u^\eps $ of \eqref{a4} 
such that $u^\eps (0,z)=z+\eps $ is calculated for every $\eps >0$. One gets
$$u^\eps (t,z)=\frac{a^\eps (t)z+b^\eps (t)}{1-p^\eps (t)z}\ .$$
Here we are only interested in the expression of $b^\eps (t)$, which is 
$$b^\eps (t)={\rm e}^{-it(1+\eps^2/2)}\left ( \eps \cos (\omega t)-\frac{2+\eps ^2}{\sqrt {4+\eps ^2}} \sin (\omega t)   \right )\ ,\ \omega :=\frac{\eps }{2}\sqrt{4+\eps ^2}\ .$$
Set 
$$T^\eps :=\frac{\pi }{2\omega }\sim \frac{\pi}{2\eps }\ .$$
Then the above formula implies
$$|b^\eps (T^\eps )|=\frac{2+\eps ^2}{\sqrt {4+\eps ^2}}\to 1\ .$$
Choose a sequence $(R^\eps )$ of positive numbers such that 
$$R^\eps \eps \to 0\ ,\ (R^\eps )^2 \eps \to +\infty \ ,$$
and a sequence $(N^\eps )$ of positive integers such that 
$$(N^\eps )^{-\delta} R^\eps \to 0\ .$$
Set 
$$v^\eps (t,z)=R^\eps u^\eps \left ((R^\eps )^2t, z^{N^\eps }\right )\ ,\ t^\eps := \frac{T^\eps }{(R^\eps )^2 } $$
Then $v^\eps $ is a solution of \eqref{a4} and 
\begin{eqnarray*}
 &&\| v^\eps (0)\| _{W^{-\delta ,2}} ^2= (N^\eps )^{-2\delta} (R^\eps)^2+(R^\eps \eps )^2\to 0\ ,\\  && |\jap {v^\eps (t^\eps ),\1}|= R^\eps |\jap{u^\eps (T^\eps ),\1}|=R^\eps |b^\eps (T^\eps )|\to +\infty \ ,
 \end{eqnarray*}
while
$$t^\eps \sim \frac{\pi}{2(R^\eps )^2 \eps }\to 0\ .$$
\end{proof}

%%%%%%%%%%%%%%%%%%%%%%%%%%%%%%%%%%%%%%%%%%%%%%%%%
%%%%%%%%%%%%%%%%%%%%%%%%%%%%%%%%%%%%%%%%%%%%%%%%%
\section{Proof of Theorem~\ref{thm.c2}}\label{sec.d}
%%%%%%%%%%%%%%%%%%%%%%%%%%%%%%%%%%%%%%%%%%%%%%%%%
%%%%%%%%%%%%%%%%%%%%%%%%%%%%%%%%%%%%%%%%%%%%%%%%%

%%%%%%%%%%%%%%%%%%%%%%%%%%
\subsection{Operator theoretic preliminaries}
%%%%%%%%%%%%%%%%%%%%%%%%%%
Let $\Sigma$ be a contraction on a Hilbert space, i.e. an operator with the norm $\leq 1$. 
The \emph{defect indices} of $\Sigma$ is the (ordered) pair of integers
$$
\bigl(\rank(I-\Sigma^*\Sigma), \rank(I-\Sigma\Sigma^*)\bigr).
$$
In particular, any unitary operator has the defect indices $(0,0)$, any isometry has the defect indices $(0,k)$ with $k\geq0$ and the shift operator $S$ has the defect indices $(0,1)$.

A contraction $\Sigma$ is called \emph{completely non-unitary} (c.n.u.), if $\Sigma$ is not unitary on any of its invariant subspaces. 
The following result is a particular case of the \emph{Wold decomposition} of an isometry (see e.g. \cite[Theorem I.1.1]{RN}).

%%%%%%%%%%%%%%%
\begin{theorem}[Wold decomposition]\label{thm.d1}
%%%%%%%%%%%%%%%
Let $\Sigma$ be an isometry on a Hilbert space $X$ with defect indices $(0,1)$, and let $q\in\Ran(I-\Sigma\Sigma^*)$, $q\not=0$. 
Then $X$ can be represented as an orthogonal sum 
$X=X_{\rm u}\oplus X_{\rm cnu}$, such that 
\begin{equation}
\Sigma=\begin{pmatrix} 
\Sigma_{\rm u} & 0
\\
0 & \Sigma_{\rm cnu}
\end{pmatrix}
\quad \text{ in $X_{\rm u}\oplus X_{\rm cnu}$,}
\label{d1}
\end{equation}
where $\Sigma_{\rm u}$ is unitary and $\Sigma_{\rm cnu}$ is c.n.u. 
Moreover, $X_{\rm cnu}$ coincides with the closed linear span of $\{\Sigma^m q\}_{m=0}^\infty$ and $\Sigma_{\rm cnu}$ is unitarily equivalent to the shift operator $S$ on $H^2$. 
\end{theorem}

We will also need the Kato-Rosenblum theorem, see e.g. \cite[Theorem~X-4.3]{Kato}.

%%%%%%%%%%%%%%%%%%%
\begin{theorem}[Kato-Rosenblum]
%%%%%%%%%%%%%%%%%%%
Let $A$ and $B$ be self-adjoint operators in a Hilbert space $X$ such that the difference $A-B$ is trace class. Then the absolutely continuous parts of $A$ and $B$ are unitarily equivalent. 
\end{theorem}
We will denote the unitary equivalence of the absolutely continuous parts of $A$ and $B$ by writing $A^{\rm ac}\simeq B^{\rm ac}$.

%%%%%%%%%%%%%%%%%%%%%%%%%%
\subsection{The operator $\Sigma$}
%%%%%%%%%%%%%%%%%%%%%%%%%%
We start by introducing an auxiliary operator. We fix $u\in H^2$ and $t\in\bbR$ for the remainder of the proof. Let
$$
\Sigma=e^{itH_u^2}Se^{-it\wt H_u^2}, \quad q=e^{itH_u^2}\1.
$$
Observe that $\Sigma$ satisfies 
\begin{equation}
\Sigma^*\Sigma=I, \quad \Sigma\Sigma^*=I-\jap{\cdot,q}q, 
\label{d2}
\end{equation}
and $\norm{q}=1$ (compare with \eqref{c1}). 
In particular, the defect indices of $\Sigma$ are $(0,1)$. 

With this notation, for any $\abs{z}<1$ we have
\begin{align*}
(I-ze^{-itH_u^2}e^{it\wt H_u^2}S^*)^{-1}e^{-itH_u^2}
&=
e^{-itH_u^2}(I-ze^{it\wt H_u^2}S^*e^{-it H_u^2})^{-1}
\\
&=
e^{-itH_u^2}(I-z\Sigma^*)^{-1},
\end{align*}
and therefore
\begin{align*}
(\Phi(t)u)(z)
&=
\jap{e^{-itH_u^2}(I-z\Sigma^*)^{-1}u,\1}
=
\jap{(I-z\Sigma^*)^{-1}u,q}
\\
&=\sum_{m=0}^\infty z^m \jap{u,\Sigma^m q}\ .
\end{align*}
Our task is to show that
\begin{equation}
\sum_{m=0}^\infty \abs{\jap{u,\Sigma^m q}}^2=\norm{u}^2.
\label{d4}
\end{equation}
From \eqref{d2} we find that $\norm{\Sigma^m q}=\norm{q}=1$ for all $m$ and $\Sigma^*q=0$. 
It follows that for $m>n\geq0$
$$
\jap{\Sigma^m q,\Sigma^n q}=\jap{\Sigma^{m-n}q,q}=\jap{q,(\Sigma^*)^{m-n}q}=0,
$$
and therefore
$\{\Sigma^m q\}_{m=0}^\infty$ is an orthonormal set in the Hardy space $H^2$. 
Thus, proving \eqref{d4} reduces to showing that $u$ belongs to the closed linear span of this orthonormal set. 

Consider the Wold decomposition $H^2=X_{\rm u}\oplus X_{\rm cnu}$ of $\Sigma$, see Theorem~\ref{thm.d1}. We see that $X_{\rm cnu}$ coincides with the closed linear span of $\{\Sigma^m q\}_{m=0}^\infty$. Thus, we need to show that $u\in X_{\rm cnu}$. 

Our plan of the proof is as follows. First we show that the unitary part $\Sigma_{\rm u}$ has no absolutely continuous part. Then, by using the previous step and a commutation relation of $\Sigma$ with an auxiliary Hankel-like operator $\calH$, we show that $u\perp X_{\rm u}$ and therefore $u\in X_{\rm cnu}$. 

\subsection{$\Sigma_{\rm u}$ has no absolutely continuous part}
Recall that by \eqref{c2} the difference $H_u^2-\wt H_u^2$ is a rank one operator. By the Duhamel formula, 
$$
e^{it\wt H_u^2}e^{-itH_u^2}-I
=
\int_0^t e^{is\wt H_u^2}(i\wt H_u^2-iH_u^2)e^{-isH_u^2}ds
=
-i\int_0^t e^{is\wt H_u^2}\bigl(\jap{\cdot,u}u\bigr)e^{-isH_u^2}ds
$$
is a trace class operator. Denoting 
$$
\Sigma_0=e^{itH_u^2}Se^{-itH_u^2}, 
$$
it follows that the operator
$$
\Sigma-\Sigma_0
=
e^{itH_u^2}Se^{-it\wt H_u^2}(I-e^{it\wt H_u^2}e^{-itH_u^2})
$$
is trace class. We will need the following abstract lemma.

\begin{lemma}\label{lma.dd1}
Let $\Sigma_0$ be a contraction on a Hilbert space $X$ such that $\Sigma_0$ is unitarily equivalent to the shift operator $S$ on $H^2$. Let $\Sigma$ be another contraction on $X$ with defect indices $(0,1)$. Assume that $\Sigma-\Sigma_0$ is trace class; then the unitary part of $\Sigma$ in the Wold decomposition \eqref{d1} has no absolutely continuous component. 
\end{lemma}
\begin{proof}
Denoting $\Re \Sigma=(\Sigma+\Sigma^*)/2$, we find that the difference
$$
\Re\Sigma-\Re \Sigma_0
$$
is a trace class operator. By the Kato-Rosenblum Theorem, we obtain
$$
(\Re\Sigma)^{\rm ac}\simeq (\Re S)^{\rm ac}. 
$$
In the standard basis of the Hardy space $H^2$ the self-adjoint operator $\Re S$ can be written as the infinite Jacobi matrix
$$
\Re S\simeq
\frac12
\begin{pmatrix}
0&1&0&0&\cdots
\\
1&0&1&0&\cdots
\\
0&1&0&1&\cdots
\\
0&0&1&0&\cdots
\\
\vdots&\vdots&\vdots&\vdots&\ddots
\end{pmatrix}\ .
$$
It is well known that the spectrum of this Jacobi matrix is purely absolutely continuous, has multiplicity one and coincides with the interval $[-1,1]$. 
Thus, we find that 
\begin{equation}
(\Re\Sigma)^{\rm ac}\simeq \Re S. 
\label{d5}
\end{equation}
On the other hand, by the Wold decomposition \eqref{d1}
$$
(\Re\Sigma)^{\rm ac}
\simeq
(\Re\Sigma_{\rm u})^{\rm ac}\oplus(\Re\Sigma_{\rm cnu})^{\rm ac},
$$
where $\Sigma_{\rm cnu}$ is unitarily equivalent to $S$. From here we find
$$
(\Re\Sigma)^{\rm ac}
\simeq
(\Re\Sigma_{\rm u})^{\rm ac}\oplus\Re S.
$$
Combining this with \eqref{d5}, we find
$$
\Re S
\simeq
(\Re\Sigma_{\rm u})^{\rm ac}\oplus\Re S\ .
$$
The multiplicity function of the spectrum of the orthogonal sum on the right hand side equals the sum of the multiplicity functions of $(\Re\Sigma_{\rm u})^{\rm ac}$ and $\Re S$. Comparing with the multiplicity function of the left hand side, we conclude that $(\Re\Sigma_{\rm u})^{\rm ac}=0$. By the functional calculus for normal operators, the spectral family of the self-adjoint operator $\Re\Sigma_{\rm u}$ can be expressed as the push-forward of the spectral family of the unitary operator $\Sigma_{\rm u}$ by the map $z\mapsto \Re z$. From here it follows that $\Sigma_{\rm u}$ has no absolutely continuous part. 
\end{proof}

\subsection{The auxiliary operator $\calH$}
Here we introduce an auxiliary anti-linear operator whose properties mirror those of $H_u$. We set 
$$
\calH=e^{itH_u^2}H_u, \quad \Dom \calH=\Dom H_u.
$$
Since $H_u$ is closed, it is clear that $\calH$ is closed. Since $H_u$ is anti-linear and since obviously $H_u$ commutes with $H_u^2$, we can also write $\calH=H_ue^{-itH_u^2}$.  Using Corollary~\ref{cr.c2}, we find 
\begin{equation}
\jap{\calH f,g}
=
\jap{e^{itH_u^2}H_uf,g}
=
\jap{H_uf,e^{-itH_u^2}g}
=
\jap{H_ue^{-itH_u^2}g,f}
=
\jap{\calH g,f}
\label{d9}
\end{equation}
for all $f,g\in\Dom \calH$.
Moreover, directly from Theorem~\ref{thm.c1}(iii) we obtain the following statement:  if for some $f,h\in H^2$ we have
$$
\jap{\calH g,f}=\jap{h,g}, \quad \forall g\in \Dom \calH,
$$
then $f\in\Dom\calH$ and $\calH f=h$.

Bearing in mind the relations \eqref{c4}, we find
$$
\Sigma^*\calH
=
(e^{it\wt H_u^2}S^*e^{-itH_u^2})(e^{itH_u^2}H_u)
=
e^{it\wt H_u^2}S^*H_u=e^{it\wt H_u^2}\wt H_u=\wt H_ue^{-it\wt H_u^2}, 
$$
and similarly 
$$
\calH \Sigma
=
(H_ue^{-itH_u^2})(e^{itH_u^2}Se^{-it\wt H_u^2})
=
H_uSe^{-it\wt H_u^2}
=
\wt H_ue^{-it\wt H_u^2}
$$
on the domain of $\calH$. It follows that $\Sigma\Dom \calH\subset \Dom \calH$ and 
\begin{equation}
\calH\Sigma=\Sigma^*\calH.
\label{d8}
\end{equation}
Further, we have
$$
\calH q=(H_ue^{-itH_u^2})(e^{itH_u^2}\1)=H_u\1=u
$$
and in particular $q\in\Dom \calH$. 

Below in Lemma~\ref{lma.dd2} we show that the Wold decomposition of $\Sigma$ reduces $\calH$, i.e. for every $f\in X_{\rm cnu}\cap\Dom \calH$ we have $\calH f\in X_{\rm cnu}$. Since $q\in X_{\rm cnu}$, this will give the desired inclusion $u=\calH q\in X_{\rm cnu}$. Thus, Theorem~\ref{thm.c2} will follow from Lemma~\ref{lma.dd2}.

\subsection{The Wold decomposition of $\Sigma$ reduces $\calH$}
For further references (and also for conceptual clarity), we prefer to state the result we need in an abstract form. 

\begin{lemma}\label{lma.dd2}
Let $\Sigma$ be a contraction on a Hilbert space $X$ satisfying the hypotheses of Lemma~\ref{lma.dd1}. Let $\calH$ be an anti-linear operator on $X$ with a dense domain $\Dom \calH$ with the following properties: 
\begin{enumerate}[\rm (a)]
\item
symmetry
\begin{equation}
\jap{\calH f,g}=\jap{\calH g,f}, \quad \forall f,g\in\Dom \calH;
\label{dd2}
\end{equation}
\item
if for some $f,h\in X$ we have 
$$
\jap{\calH g,f}=\jap{h,g}, \quad \forall g\in\Dom\calH,
$$
then $f\in\Dom\calH$ and $\calH f=h$;
\item
$\Sigma\Dom\calH\subset\Dom\calH$ and 
\begin{equation}
\Sigma^*\calH=\calH\Sigma
\label{dd1}
\end{equation}
on the domain of $\calH$. 
\end{enumerate}
Then the Wold decomposition \eqref{d1} of $\Sigma$ reduces $\calH$, i.e. if 
$P_{\rm cnu}$  is the orthogonal projection in $X$ onto the subspace $X_{\rm cnu}$ in \eqref{d1}, then $P_{\rm cnu}\Dom\calH\subset\Dom\calH$ and $\calH P_{\rm cnu}(\Dom\calH)\subset X_{\rm cnu}$. 
\end{lemma}
\begin{proof}
First we need to discuss the spectral measures associated with $\Sigma$.
Let $P_{\rm u}$ (resp. $P_{\rm cnu}$) be the orthogonal projection in $X$ onto $X_{\rm u}$ (resp. $X_{\rm cnu}$) in the Wold decomposition \eqref{d1} of $\Sigma$. For $f,g\in X$ and $m\geq0$, consider
\begin{equation}
\jap{\Sigma^m f,g}
=
\jap{\Sigma_{\rm u}^m P_{\rm u}f,P_{\rm u}g}
+
\jap{\Sigma_{\rm cnu}^m P_{\rm cnu}f,P_{\rm cnu}g}.
\label{d6}
\end{equation}
We can write 
\begin{equation}
\jap{\Sigma_{\rm u}^m P_{\rm u}f,P_{\rm u}g}
=
\int_{\bbT} z^md\nu_{f,g}(z),
\label{d6a}
\end{equation}
where $\nu_{f,g}$ is the spectral measure of the unitary operator $\Sigma_{\rm u}$, associated with the pair of vectors $P_{\rm u}f,P_{\rm u}g$. By Lemma~\ref{lma.dd1}, the measure $\nu_{f,g}$ is singular with respect to the Lebesgue measure. 

In order to consider the second term in the r.h.s. of \eqref{d6}, we recall that by Theorem~\ref{thm.d1}, $\Sigma_{\rm cnu}$ is unitarily equivalent to the shift operator $S$ in the Hardy space $H^2$. The unitary map that effects this equivalence maps $P_{\rm cnu}f$ and $P_{\rm cnu}g$ to some elements $f_*,g_*\in H^2$ and so we can write
$$
\jap{\Sigma_{\rm cnu}^m P_{\rm cnu}f,P_{\rm cnu}g}
=
\int_{\bbT}z^m f_*(z)\overline{g_*(z)}d\lambda(z)
=
\int_{\bbT}z^m d\mu_{f,g}(z),
$$
where $d\mu_{f,g}$ is the absolutely continuous measure on $\bbT$ defined by the above identity. Putting this together, we find
\begin{equation}
\jap{\Sigma^m f,g}
=
\int_{\bbT} z^md\nu_{f,g}(z)
+
\int_{\bbT}z^m d\mu_{f,g}(z),\quad m\geq0,
\label{d7}
\end{equation}
where $d\nu_{f,g}$ is singular and $d\mu_{f,g}$ is absolutely continuous.

Iterating the commutation relation \eqref{dd1}, we find
$$
\calH \Sigma^m=(\Sigma^*)^m \calH
$$
for all $m\geq0$. 
Taking the bilinear form of this relation on elements $f,g\in\Dom\calH$, we find
$$
\jap{\calH \Sigma^m f,g}=\jap{(\Sigma^*)^m \calH f,g}=\jap{\calH f,\Sigma^m g}.
$$
Using the symmetry relation \eqref{dd2}, the left hand side here rewrites as
$$
\jap{\calH \Sigma^m f,g}=\jap{\calH g, \Sigma^m f}. 
$$
Putting this together and taking complex conjugates, we find 
$$
\jap{\Sigma^m f,\calH g}=\jap{\Sigma^m g,\calH f}. 
$$
By \eqref{d7}, this yields
$$
\int_{\bbT}z^{m}d\nu_{f,\calH g}(z)
+
\int_{\bbT}z^{m}d\mu_{f,\calH g}(z)
=
\int_{\bbT}z^{m}d\nu_{g,\calH f}(z)
+
\int_{\bbT}z^{m}d\mu_{g,\calH f}(z)
$$
for all $m\geq0$. By the F. and M. Riesz theorem, we obtain the equality of the singular components of the measure:
$$
\nu_{f,\calH g}=\nu_{g,\calH f} \ . 
$$
In particular, 
$$
\nu_{f,\calH g}(\bbT)=\nu_{g,\calH f}(\bbT);
$$
recalling the definition \eqref{d6a} of the measure $\nu_{f,g}$, we find
$$
\jap{P_{\rm u}f,P_{\rm u}\calH g}=\jap{P_{\rm u}g,P_{\rm u}\calH f}
$$
or equivalently (using that $P_{\rm u}^2=P_{\rm u}$ and taking complex conjugates)
$$
\jap{\calH g,P_{\rm u}f}=\jap{P_{\rm u}\calH f,g}.
$$
Now let us fix $f\in\Dom\calH$ and consider the above relation for all $g\in\Dom \calH$. 
By our hypothesis (b), this implies that $P_{\rm u}f\in\Dom\calH$ and $\calH P_{\rm u}f=P_{\rm u}\calH f$. 

Since $f=P_{\rm u}f+P_{\rm cnu}f$, by the linearity of $\Dom \calH$ we also have $P_{\rm cnu}f\in\Dom\calH$, and furthermore  $\calH P_{\rm cnu}f=P_{\rm cnu}\calH f$. 
\end{proof}

\appendix

%%%%%%%%%%%%%%%%%%%%%%%%%%%%
%%%%%%%%%%%%%%%%%%%%%%%%%%%%
\section{Toeplitz operator with $[T_\varphi ]_{\rm max}\ne [T_\varphi ]_{\rm min}$}
%%%%%%%%%%%%%%%%%%%%%%%%%%%%
%%%%%%%%%%%%%%%%%%%%%%%%%%%%
In this appendix, we reproduce the construction of Hartman--Wintner \cite{HaWi}, giving a Toeplitz operator with an unbounded real valued symbol $\varphi \in L^2(\bbT)$ such that $[T_\varphi ]_{\rm max}\ne [T_\varphi ]_{\rm min}$. 
\begin{lemma}\label{A1}
Let $\varphi \in L^2(\bbT)$ be real valued. If $[T_\varphi ]_{\rm max}= [T_\varphi ]_{\rm min}$, then this operator is selfadjoint.
\end{lemma}
\begin{proof}
Since the operator $[T_\varphi ]_{\rm min}$ is symmetric, we have $[T_\varphi ]_{\rm min}\subset [T_\varphi ]_{\rm min}^*$. On the other hand, it is easy to see directly that $[T_\varphi ]_{\rm min}^*\subset [T_\varphi ]_{\rm max}$. If $[T_\varphi ]_{\rm max}= [T_\varphi ]_{\rm min}$, we obtain self-adjointness. 
\end{proof}

\begin{lemma}\label{A2}
Let $\varphi \in L^2(\bbT)$ be real valued such that $[T_\varphi ]_{\rm max}= [T_\varphi ]_{\rm min}$. Then the kernel of this operator is trivial, unless $\varphi$ is identically equal to zero.
\end{lemma}
\begin{proof}
We note that if $\varphi$ is bounded, the statement is a consequence of Coburn's lemma \cite[Section 4]{Coburn}.  In the general case $\varphi \in L^2(\bbT)$, we give the proof as in the original paper of Hartmann and Wintner \cite{HaWi}. Denote $T_\varphi:=[T_\varphi ]_{\rm max}$ and suppose $T_\varphi f=0$ for some $f\in\Dom T_\varphi$ and $\varphi$ is not identically zero. This means that  $\varphi f=g$ on the unit circle with some $g\in L^1(\bbT)$ such that
$$
\varphi(z)f(z)=g_1\overline{z}+g_2\overline{z}^2+g_3\overline{z}^3+\cdots;
$$
observe that $g$ is not identically zero, because $f$ is non-zero on the unit circle almost everywhere. 
Let $g_N$ be the first non-zero coefficient in the right hand side; by renormalising, we may assume $g_N=1$. Then $T_\varphi(z^Nf)=\1$, and so $\1\in\Ran T_\varphi$. Similarly, 
$$
T_\varphi(z^{N+1}f)=z+g_{N+1}\1,
$$
and so $z\in\Ran T_\varphi$. Arguing by induction, we find that $\1,z,z^2,\dots\in\Ran T_\varphi$, and therefore $\Ran T_\varphi$ is dense in $H^2$. On the other hand, by the previous Lemma, $T_\varphi$ is self-adjoint and therefore 
$(\Ran T_\varphi)^\perp=\Ker T_\varphi\not=\{0\}$; this is a contradiction. 
\end{proof}

Now we construct a real valued function $\varphi $ in $L^2(\bbT)$ such that the kernel of $[T_\varphi ]_{\rm max}$ is not trivial. By  Lemmas \ref{A1} and \ref{A2}, this implies  that $[T_\varphi ]_{\rm max}\ne [T_\varphi ]_{\rm min}$. 
\vskip 0.25cm
We fix $\varepsilon \in (0,\tfrac12)$ and  denote by $F, H$ the holomorphic functions on $\bbD$ defined by 
$$F(z)=\left ( \frac{i+z}{i-z}  \right )^{1/2}\ ,\ H(z)= [(z-i)(z+i)]^{\varepsilon /2}\ ,\ F(0)=1\ ,\ H(0)=1\ .$$
Then we define functions $\varphi_1$, $\varphi_2$ on the unit circle by 
$$
\varphi_1\left ({\rm e}^{i \theta }\right )=
\begin{cases} 
-1\ & {\rm if}\ |\theta |<\pi /2\\
1 \ & {\rm if}\ \pi/2<|\theta |<\pi 
\end{cases}
$$
and 
$$\varphi_2\left ({\rm e}^{i \theta }\right )=\frac{1}{| ( {\rm e}^{i \theta }-i)({\rm e}^{i \theta }+i)   |^\varepsilon }\ .$$
Denote $\varphi =\varphi_1\varphi_2$. It is clear that $\varphi \in L^2(\bbT)$; this will be the desired symbol.
\begin{lemma}\label{A3}
We have the following identities on $\bbT$,
\begin{eqnarray*}
-i\varphi _1F&=&1+{\rm e}^{-i \theta }\overline g_1\ ,\ g_1\in \bigcap_{p<2}H^p\ ,\\
i \frac{\varphi_1}{F}&=&1+{\rm e}^{-i \theta }\overline h_1\ ,\ h_1\in \bigcap_{p<2}H^p\ ,\\
\varphi_2H&=&1+{\rm e}^{-i \theta }\overline g_2\ ,\ g_2\in H^2\ .
\end{eqnarray*}
\end{lemma}
\begin{proof} First observe that $F\in H^p$ for every $p<2$ and that $F^2$ is purely imaginary on the unit circle.  On $\bbT$, we define $u:={\rm Re} (F)$, $v:={\rm Im} (F)$, so that $u^2=v^2$. Furthermore, it is easy to check that $\varphi_1$ is the sign of $uv$. Consequently, we have
$$\varphi_1(v-iu)=(u-iv)\ ,\ \varphi_1(v+iu)=u+iv\ .$$
The first identity can be rephrased as 
$$-i\varphi _1F=\overline F\ .$$
In view of the expression of $F$, $$\overline F=1+{\rm e}^{-i \theta }\overline g_1\ ,\ g_1\in \bigcap_{p<2}H^p\ .$$
As for the second identity, we divide both sides by $u^2+v^2$ and obtain
$$i \frac{\varphi_1}{F}=\frac{1}{\overline F}=1+{\rm e}^{-i \theta }\overline h_1\ ,\ h_1\in \bigcap_{p<2}H^p\ . $$
Let us check the third identity. We have
$$
\varphi_2H({\rm e}^{i \theta })=\frac{[({\rm e}^{i \theta }-i)({\rm e}^{i \theta }+i)]^{\varepsilon /2}}{| ( {\rm e}^{i \theta }-i)({\rm e}^{i \theta }+i)   |^\varepsilon }
=
\frac{1}{[({\rm e}^{-i \theta }+i)({\rm e}^{-i \theta }-i)]^{\varepsilon /2}}
=
1+{\rm e}^{-i \theta }\overline g_2\ ,
$$
for some $g_2\in H^2$, as required. 
\end{proof}
Let us complete the construction. First of all, notice that 
$$
FH\ \in H^2  \quad\text{ and }\quad \frac{H}{F} \in H^2\ .
$$
We set
$$f:=FH+\frac{H}{F}\in H^2\ .$$
Then $f$ is not identically $0$ and, from Lemma \ref{A3},
\begin{eqnarray*}
 \varphi f&=& \varphi_1 F \varphi_2H+ \frac{\varphi_1}{F}\varphi_2H \\
 &=&i(1+{\rm e}^{-i \theta }\overline g_1)(1+{\rm e}^{-i \theta }\overline g_2)-i(1+{\rm e}^{-i \theta }\overline h_1)(1+{\rm e}^{-i \theta }\overline g_2)\ .
\end{eqnarray*}
 As a consequence, $\mathbb P (\varphi f)=0$ in the sense of distributions. Hence $f$ is a nontrivial vector of the kernel of $[T_\varphi]_{\rm max}$. The proof is complete.


\begin{thebibliography}{2}

\bibitem{Coburn}
{\sc L.~A.~Coburn,}
\emph{Weyl's theorem for nonnormal operators,}
Michigan Math. J. \textbf{13} no.3 (1966), 285--288.

\bibitem{GG00}
{\sc P.~G\'erard, S.~Grellier, }
\emph{The cubic Szeg\H{o} equation}, 
Ann. Scient. \'Ec. Norm. Sup. \textbf{43} (2010), 761--810.

\bibitem{GG12}
{\sc P.~G\'erard, S.~Grellier, }
\emph{Invariant tori for the cubic Szego
equation,} Invent. Math., \textbf{187} (2012), 707--754.

\bibitem{GGhalf}
{\sc P.~G\'erard, S.~Grellier, }
\emph{Effective integrable dynamics
for a certain nonlinear wave equation},  Anal. PDE, \textbf{43} (2012), 1139--1155.

\bibitem{GG2}
{\sc P.~G\'erard, S.~Grellier, }
\emph{An explicit formula for the cubic Szeg\H{o} equation},
Trans. Amer. Math. Soc. \textbf{367} (2015), 2979--2995.

\bibitem{GG17}
{\sc P.~G\'erard, S.~Grellier, }
\emph{The cubic Szeg\H{o} equation and Hankel operators,}
Ast\'erisque \textbf{389} (2017). 

\bibitem{GK}
{\sc P. G\'erard, H. Koch,}
\emph{The cubic Szeg\H{o} flow at low regularity},
S\'eminaire Laurent Schwartz, EDP et Applications, 2016--2017, \textbf{XIV}, 14 p.


\bibitem{GKT} 
{\sc P.~G\'erard, T.~Kappeler, P.~Topalov,} 
\emph{Sharp wellposedness results for the Benjamin--Ono   in $H^s(\bbT,\bbR)$ and qualitative properties of its solution} , 
{\it Preprint}, 2020, {\tt arXiv:2004.04857}, to appear in Acta Math.

\bibitem{GPT}
{\sc P.~G\'erard, A.~Pushnitski, S.~Treil,}
\emph{An inverse spectral problem for non-compact Hankel operators with simple spectrum,}
arXiv:2211.00965. 

\bibitem{HKV}
{\sc B.~Harrop--Griffiths, R.~Killip, M.~Vi\c san,} 
\emph{Sharp wellposedness for the cubic NLS and MKdV in $H^s(\bbR)$}, 
{\it Preprint}, 2020, {\tt arXiv:2003.05011}.


\bibitem{HKNV} 
{\sc B.~Harrop--Griffiths, R.~Killip, M. Ntekoume, M.~Vi\c san,}
\emph{Global wellposedness for the derivative nonlinear Schr\"odinger equation on $L^2(\bbR )$},
 {\it Preprint}, 2021,  {\tt arXiv:2204.12548}.
 
 \bibitem{HaWi}
 {\sc P.~Hartman, A.~Wintner,}
 \emph{The Spectra of Toeplitz's Matrices}, 
 Amer. J. Math. \textbf{76}, no.4 (1954), 867--882


\bibitem{Kato}
{\sc T.~Kato,}
\emph{Perturbation theory for linear operators,} 
Springer-Verlag, New York 1966.

\bibitem{Peller}
{\sc V.~Peller,}
\emph{Hankel operators and their applications,}
Springer, 2003.

\bibitem{RN}
{\sc B.~Sz.-Nagy,  C.~Foias,  H.~Bercovici, L.~K\'erchy,}
\emph{Harmonic analysis of operators on Hilbert space,}
Second edition. Revised and enlarged edition. Universitext. Springer, New York, 2010.

\bibitem{RS1}
{\sc M.~Reed, B.~Simon,}
\emph{Methods of modern mathematical physics. I: Functional analysis,}
Revised and enlarged edition,
Academic Press, 1980.


\bibitem{RS2}
{\sc M.~Reed, B.~Simon,}
\emph{Methods of modern mathematical physics. II: Fourier analysis, selfadjointness,}
Revised and enlarged edition, Academic Press, 1980.

\bibitem{Simon}
{\sc B.~Simon,}
\emph{The Classical Moment Problem as a Self-Adjoint Finite Difference Operator,}
Adv. in Math. \textbf{137} (1998), 82--203.


\bibitem{Timotin}
{\sc D.~Timotin,}
\emph{The invariant subspaces of $S\oplus S^*$,}
Concr. Oper.  \textbf{7} (2020), 116--123.

\bibitem{Xu}
{\sc H. Xu,} \emph{Unbounded Sobolev trajectories and   modified scattering theory for a wave guide nonlinear Schr\"odinger equation, }
Math. Z. \textbf{286} (2017), 443--489.

\bibitem{Yafaev1}
{\sc D.~Yafaev,}
\emph{Unbounded Hankel operators and moment problems,} Integral Equations Operator Theory \textbf{85} (2016), no. 2, 289--300.

\bibitem{Yafaev2}
{\sc D.~Yafaev,}
\emph{Correction to: Unbounded Hankel operators and moment problems,} 
Integral Equations Operator Theory \textbf{91} (2019), no. 5, Paper No. 44, 4 pp.

\end{thebibliography}
\end{document}